\pdfoutput=1
\documentclass[a4paper, 11pt, reqno]{amsart}

\usepackage{amssymb}
\usepackage{amsmath}
\usepackage{amsthm}
\usepackage[numbers]{natbib}
\usepackage{mathtools}
\usepackage{enumitem}
\usepackage{calc}
\usepackage{setspace}
\usepackage{graphicx}
\usepackage{tikz-cd}
\usepackage{hyperref}
\usepackage[top=3 cm, bottom=3 cm, left=3 cm, right=3cm]{geometry}
\onehalfspace

\newcounter{dummy}
\numberwithin{dummy}{section}
\newtheorem{thm}[dummy]{Theorem}
\newtheorem{defn}[dummy]{Definition}
\newtheorem{conj}[dummy]{Conjecture}
\newtheorem{lem}[dummy]{Lemma}
\newtheorem{prop}[dummy]{Proposition}
\newtheorem{cor}[dummy]{Corollary}
\newtheorem{claim}{Claim}
\numberwithin{equation}{section}

\newcommand{\noop}[1]{}

\DeclareMathOperator{\Gal}{\mathrm{Gal}}

\DeclareMathOperator{\tors}{\mathrm{tors}}
\DeclareMathOperator{\der}{\mathrm{der}}
\DeclareMathOperator{\Tr}{\mathrm{Tr}}

\DeclareMathOperator{\GSp}{\mathrm{GSp}}

\newcommand{\ZP}{\text{ZP}}

\makeatletter
\renewcommand{\@biblabel}[1]{[#1]\hfill}
\makeatother

\begin{document}
\title{On the Zilber-Pink conjecture for complex abelian varieties}

\author[F. Barroero]{Fabrizio Barroero}
\address{Universit\`a degli studi Roma Tre, Dipartimento di Matematica e Fisica, Largo San Murialdo 1, 00146 Roma, Italy}
\email{fbarroero@gmail.com}

\author[G. A. Dill]{Gabriel A. Dill}
\address{Universit\"at Basel, Departement Mathematik und Informatik, Spiegelgasse 1, CH-4051 Basel, Switzerland}
\email{gabriel.dill@unibas.ch}

\date{\today}

\begin{abstract}
In this article, we prove that the Zilber-Pink conjecture for abelian varieties over an arbitrary field of characteristic $0$ is implied by the same statement for abelian varieties over the algebraic numbers.

More precisely, the conjecture holds for subvarieties of dimension at most $m$ in the abelian variety $A$ if it holds for subvarieties of dimension at most $m$ in the largest abelian subvariety of $A$ that is isomorphic to an abelian variety defined over $\bar{ \mathbb{Q}}$.

Dans cet article, nous prouvons que la conjecture de Zilber-Pink pour les vari\'et\'es ab\'eliennes sur un corps quelconque de caracteristique $0$ est impliqu\'ee par le m\^eme \'enonc\'e pour les vari\'et\'es ab\'eliennes sur les nombres alg\'ebriques.

Plus pr\'ecis\'ement, la conjecture est vraie pour les sous-vari\'et\'es de dimension inf\'erieure ou \'egale \`a $m$ dans la vari\'et\'e ab\'elienne $A$ si elle est vraie pour les sous-vari\'et\'es de dimension inf\'erieure ou \'egale \`a $m$ dans la plus grande sous-vari\'et\'e ab\'elienne de $A$ qui est isomorphe \`a une vari\'et\'e ab\'elienne d\'efinie sur $\bar{\mathbb{Q}}$.
\end{abstract}

\maketitle

\section{Introduction}

For us, varieties and curves are irreducible and subvarieties are always irreducible and closed in the ambient variety. Fields are always of characteristic 0. We work with the Zariski topology, therefore by open, dense, etc.~we always mean Zariski open, Zariski dense, etc.~except when we consider connected mixed Shimura (sub)varieties or (sub)data.

Let $A$ be an abelian variety defined over an algebraically closed field $K$. A special subvariety of $A$ is an irreducible component of an algebraic subgroup of $A$ or, equivalently, a translate of an abelian subvariety by a torsion point. Arbitrary translates of abelian subvarieties are called cosets or weakly special subvarieties. Special subvarieties are also called torsion cosets.

The Manin-Mumford conjecture, proven by Raynaud \cite{Raynaud}, states that a subvariety of an abelian variety contains at most finitely many maximal special subvarieties. In particular, a non-special curve contains at most finitely many torsion points.

On the other hand, given a curve in an abelian variety, a dimension count suggests that it should not intersect a special subvariety of codimension at least 2. If one considers the union of all special subvarieties of codimension at least 2 and intersects it with a curve that is not contained in a proper special subvariety, one expects the intersection to be finite.

The pioneering work \cite{BMZ99} of Bombieri, Masser and Zannier was one of the first to study this kind of problems and to pass from considering torsion points in subvarieties of algebraic groups to points lying in algebraic subgroups of appropriate codimension.

Indeed, Bombieri, Masser and Zannier proved that, given a curve defined over the algebraic numbers and contained in $\mathbb{G}_m^n$ but not in any of its proper (not necessarily torsion) cosets, it contains at most finitely many points that lie in an algebraic subgroup of $\mathbb{G}_m^n$ of codimension at least 2. The condition of not being contained in a proper coset was replaced by the necessary one of not lying in a proper torsion coset by Maurin \cite{Maurin} and independently by Bombieri, Habegger, Masser and Zannier in \cite{BHMZ}.

In the same paper \cite{BMZ99}, Bombieri, Masser and Zannier suggest that a possible analogue of their result for curves in $\mathbb{G}_m^n$ could hold for (families of) abelian varieties and that one could consider higher dimensional subvarieties and intersect them with algebraic subgroups of higher codimension.

A couple of years later, Zilber \cite{MR1875133} independently stated a conjecture for semiabelian varieties of which the result of Bombieri, Masser and Zannier is a consequence. This is formulated in slightly different language and we are going to state it later. Similar conjectures for $\mathbb{G}_m^n$ were formulated by Bombieri, Masser and Zannier in \cite{BMZ07}.

We now consider an apparently weaker formulation of the same principle due to Pink. We introduce the following notation: For a non-negative integer $k$, we denote by $A^{[k]}$ the union of all special subvarieties of $A$ of codimension at least $k$.

Pink conjectured in \cite{PUnpubl} that, if $V \cap A^{[\dim V +1]}$ is Zariski dense in $V$ for a subvariety $V$ of $A$, then $V$ is contained in a proper special subvariety of $A$. The conjecture in its full generality is still open. If $V$ is a curve and $K = \bar{\mathbb{Q}}$, it has been proven by Habegger and Pila in \cite{MR3552014}. Previously, partial results have been obtained by Viada \cite{MR1990974}, \cite{Viada2008}, R\'emond and Viada \cite{RemVia}, Ratazzi \cite{Ratazzi08}, Carrizosa \cite{MR2473296}, \cite{MR2533796} in combination with R\'emond \cite{MR2198798}, \cite{MR2311666}, \cite{MR2534482}, and Galateau \cite{galateau2010}. If $V$ is a hypersurface, Pink's conjecture follows from the Manin-Mumford conjecture. If the dimension and codimension of $V$ are at least $2$, then all known results place additional restrictions on $V$ or $A$, see for instance \cite{CVV}, \cite{CV}, and \cite{HubVia}.

In this article, we use a recent result of Gao in \cite{G18}, which generalizes work by R\'emond in \cite{MR2534482}, to reduce Zilber's conjecture to the case where everything is defined over $\bar{\mathbb{Q}}$. We even show that it can be reduced to Pink's formulation of the conjecture over $\bar{\mathbb{Q}}$. Furthermore, we prove the full conjecture in Corollary \ref{cor:corollary2} if no abelian variety of dimension greater than $4$ that is defined over $\bar{\mathbb{Q}}$ embeds into $A$. For example, the conjecture holds in a power of an elliptic curve with transcendental $j$-invariant. Combining Theorem \ref{thm:functionfieldzp} below with Theorem 1.1 in \cite{MR3552014} yields the following theorem:

\begin{thm}
Let $A$ be an abelian variety defined over an algebraically closed field $K$ (of characteristic $0$) and let $V \subset A$ be a curve. Then $V \cap A^{[2]}$ is finite unless $V$ is contained in a proper algebraic subgroup of $A$.
\end{thm}

As mentioned before, Pink's conjecture is implied by the following Conjecture \ref{conj:zilberpinkatyp} on unlikely or atypical intersections that was formulated by Zilber in \cite{MR1875133} for semiabelian varieties. An overview of the topic of unlikely intersections is given in the book \cite{Zannier}.

In order to state Conjecture \ref{conj:zilberpinkatyp}, we introduce the notion of an atypical subvariety: Let $A$ be an abelian variety defined over an algebraically closed field $K$ and let $V$ be a subvariety of $A$. A subvariety $W$ of $V$ is called atypical (for $V$ in $A$) if $W$ is an irreducible component of the intersection of $V$ with a special subvariety of codimension at least $\dim V - \dim W + 1$. It is called maximal if it is not contained in any larger atypical subvariety. 

\begin{conj}\label{conj:zilberpinkatyp}
Let $K$ be an algebraically closed field. Let $A$ be an abelian variety defined over $K$ and let $V$ be a subvariety of $A$. Then $V$ contains at most finitely many maximal atypical subvarieties.
\end{conj}

If $V$ is a curve, then Conjecture \ref{conj:zilberpinkatyp} and Pink's conjecture are obviously equivalent.

It turns out that another equivalent formulation of Conjecture \ref{conj:zilberpinkatyp} is more suited to our proof strategy. In order to state it, we have to introduce the notions of defect and optimality of a subvariety.

\begin{defn}\label{def:optimal}
 If $V$ is a subvariety of $A$, then there is a smallest special subvariety $\langle V\rangle$ containing $V$. We define the defect $\delta(V)$ of $V$ to be $\dim \langle V \rangle - \dim V$. A subvariety $W$ of $V$ is called optimal for $V$ in $A$ if $\delta(U) > \delta(W)$ for every subvariety $U$ with $W \subsetneq U \subset V$. 
\end{defn}

Pink introduced the notion of defect in \cite{PUnpubl}, while the concept of optimality was introduced in \cite{MR3552014} by Habegger and Pila. The latter is motivated by Poizat's notion of $cd$-maximality in \cite{Poizat}. $cd$-maximality is the toric analogue of the notion of geodesic optimality, which we will introduce later. Using the concept of optimality, Habegger and Pila formulated the following conjecture, which is equivalent to Conjecture \ref{conj:zilberpinkatyp} by Lemma 2.7 in \cite{MR3552014}. 

\begin{conj}\label{conj:zilberpink}
Let $K$ be an algebraically closed field and let $d$ be a non-negative integer. Let $A$ be an abelian variety defined over $K$ and let $V$ be a subvariety of $A$. Then $V$ contains at most finitely many optimal subvarieties of defect at most $d$.
\end{conj}

In the statement of our results, we use the trace of an abelian variety with respect to a field extension of algebraically closed fields. This can be thought of as the largest abelian subvariety defined over the smaller field. See Definition \ref{def:trace} for a formal definition.

The following is the main result of this article:

\begin{thm}\label{thm:functionfieldzp}
	Let $K$ be an algebraically closed field, let $m$ be a non-negative integer and $A$ an abelian variety defined over $K$ with $K / \bar{ \mathbb{Q}}$-trace $(T,\Tr)$. Then, if Conjecture \ref{conj:zilberpink} holds for some non-negative integer $d$ and subvarieties of dimension at most $m$ in $T$ (over the field $\bar{\mathbb{Q}}$), it holds for the same $d$ and subvarieties of dimension at most $m$ in $A$ (over $K$).
	\end{thm}

Note that Habegger and Pila have shown in Corollary 9.10 in \cite{MR3552014} that Conjecture \ref{conj:zilberpink} can be further reduced to the existence of sufficiently strong lower bounds for the size of the Galois orbits of optimal singletons over a field of definition that is finitely generated over $\mathbb{Q}$.

An analogue of Theorem \ref{thm:functionfieldzp} for powers of the multiplicative group was proven in \cite{MR2457263} by Bombieri, Masser and Zannier. Note that in this case the ambient algebraic group is always defined over $\bar{\mathbb{Q}}$. In our situation, this corresponds to the special case where $A$ is isomorphic to the base change of an abelian variety over $\bar{\mathbb{Q}}$.

Following a suggestion of Habegger, we prove Conjecture \ref{conj:zilberpink} in Theorem \ref{thm:habeggerpila} if $K = \bar{\mathbb{Q}}$ and $d = 1$. This together with the preceding Theorem \ref{thm:functionfieldzp} and Theorem 1.1 in \cite{MR3552014} implies the following corollary:

\begin{cor}\label{cor:corollary}
	Let $K$ be an algebraically closed field, let $m$ be a non-negative integer and $A$ an abelian variety defined over $K$. Then, Conjecture \ref{conj:zilberpink} holds for subvarieties of dimension at most $m$ in $A$ if either $m \leq 1$ or $d \leq 1$.
	\end{cor}
	
 Note that Conjecture \ref{conj:zilberpink} trivially holds for subvarieties of $A$ of dimension or codimension $0$. In codimension $1$, every proper optimal subvariety is special, so Conjecture \ref{conj:zilberpink} follows from the theorem of Raynaud in \cite{Raynaud} (Manin-Mumford conjecture). By Corollary \ref{cor:corollary}, Conjecture \ref{conj:zilberpink} also holds for subvarieties of codimension $2$ since every proper optimal subvariety of a subvariety of codimension $2$ has defect at most $1$; for the toric analogue see \cite{BMZ07} (there proven over $\bar{\mathbb{Q}}$, then extended to $\mathbb{C}$ in \cite{MR2457263}). Conjecture \ref{conj:zilberpink} has previously been proven for $K = \bar{\mathbb{Q}}$ and subvarieties of codimension $2$ in powers of elliptic curves with complex multiplication (CM) (in \cite{CVV}) and without CM (in \cite{HubVia}) as well as in arbitrary products of elliptic curves with CM (in \cite{CV}).
 
  In particular, Conjecture \ref{conj:zilberpink} holds for abelian varieties of dimension at most $4$ and by applying Theorem \ref{thm:functionfieldzp} we obtain the following corollary:

	\begin{cor}\label{cor:corollary2}
	Let $K$ be an algebraically closed field and $A$ an abelian variety defined over $K$ with $K / \bar{ \mathbb{Q}}$-trace $(T,\Tr)$. If $\dim T \leq 4$, then Conjecture \ref{conj:zilberpink} holds for $A$.
	\end{cor}

In the proof of Theorem \ref{thm:functionfieldzp}, we use a double induction firstly on the dimension of $A$ and secondly on the transcendence degree of its field of definition. If the transcendence degree of the field of definition is minimal in the sense that $A$ is obtained as a geometric fiber of a certain universal family $\mathfrak{A}_{g,l} \to A_{g,l}$ of abelian varieties, then we apply Gao's result to reduce to abelian varieties of smaller dimension in Proposition \ref{prop:mixedaxschanuel}.

We then use R\'emond's results to increase the transcendence degree of the field of definition in Proposition \ref{prop:fieldofdefofcurve}. This part of the proof at some points resembles the proof of the main result in \cite{MR2457263}, albeit formulated rather differently.

The proofs of both propositions begin with the use of the fact that optimal subvarieties are geodesic-optimal, i.e., optimal with respect to the geodesic defect, which is the analogue of the defect if one replaces special by weakly special subvarieties. For abelian varieties, this has been proven by Habegger and Pila. It turns out that their proof can be adapted to show that the same holds if one considers a slightly different definition of the geodesic defect, where one replaces weakly special subvarieties by translates of abelian subvarieties by a torsion point plus a $\bar{\mathbb{Q}}$-point of the trace. We call it the $\bar{\mathbb{Q}}$-geodesic defect and $\bar{\mathbb{Q}}$-geodesic-optimal subvarieties are then the analogue of geodesic-optimal subvarieties for this defect.

To any ($\bar{\mathbb{Q}}$-)geodesic-optimal subvariety, there is an associated abelian subvariety. Thanks to the results of Gao and R\'emond, this abelian subvariety lies in a finite set. If its dimension is positive, we can quotient out by it and use the inductive hypothesis. Otherwise, we either use the full strength of Gao's result to reduce to the trace or we use the inductive hypothesis on the transcendence degree of the field of definition concluding the proof.

As mentioned above, we also show that Conjecture \ref{conj:zilberpink} over an arbitrary algebraically closed field $K$ can be reduced to Pink's formulation of the conjecture over $\bar{\mathbb{Q}}$. For the precise statement of our result, we now give a more articulated version of Pink's conjecture. Recall that we denote by $A^{[k]}$ the union of all special subvarieties of an abelian variety $A$ of codimension at least $k$.

\begin{conj}\label{conj:pink}
Let $K$ be an algebraically closed field and let $d$ be a non-negative integer. Let $A$ be an abelian variety defined over $K$ and let $V$ be a subvariety of $A$. If $V \cap A^{[\max\{\dim V + 1, \dim A -d\}]}$ is Zariski dense in $V$, then $V$ is contained in a proper special subvariety of $A$.
\end{conj}

The following statement draws a link between the above conjecture and Conjecture \ref{conj:zilberpink}:

\begin{thm}\label{thm:pequalzp}
	Let $K$ be an algebraically closed field, let $m$ be a non-negative integer and $A$ an abelian variety defined over $K$. Then, if Conjecture \ref{conj:pink} holds for some non-negative integer $d$ and subvarieties of dimension at most $m$ in every abelian subvariety $B$ of $A$, Conjecture \ref{conj:zilberpink} holds for the same $d$ and subvarieties of dimension at most $m$ in $A$.
	\end{thm}
	
	We prove Theorem \ref{thm:pequalzp} in Section \ref{sec:pequalzp}. The proof is a direct application of Theorem 9.8(i) in \cite{MR3552014}. As pointed out by Ullmo and Zannier, the analogous reduction can be done in the toric case by imposing additional multiplicative relations on the positive-dimensional atypical intersections. Combining Theorem \ref{thm:pequalzp} and Theorem \ref{thm:functionfieldzp} yields the following corollary:
	
		\begin{cor}
Let $K$ be an algebraically closed field, let $m$ be a non-negative integer and $A$ an abelian variety defined over $K$ with $K / \bar{ \mathbb{Q}}$-trace $(T,\Tr)$. Then, if Conjecture \ref{conj:pink} holds for some non-negative integer $d$ and subvarieties of dimension at most $m$ in every abelian subvariety $T'$ of $T$ (over the field $\bar{\mathbb{Q}}$), Conjecture \ref{conj:zilberpink} holds for the same $d$ and subvarieties of dimension at most $m$ in $A$ (over $K$).
	\end{cor}

\section{Preliminaries}

In this section, we collect some results that are going to be useful in the proof of Theorem \ref{thm:functionfieldzp}.

\subsection{Definitions and a useful lemma}

We are going to perform several base changes. We use the following notation:

\begin{defn}\label{def:basechange}
Let $V$ be a variety over a field $K$ and let $K \subset L$ be a field extension. Then $V_L = V\times_{K} L$ is called the base change of $V$ to $L$. We use analogous notation for the base change of morphisms between varieties.
\end{defn}

In the proof of Proposition \ref{prop:fieldofdefofcurve}, we argue by induction on the transcendence degree of our field of definition. Here is a basic fact used in the inductive step.

\begin{lem}\label{lem:ausefullemma}
	Let $K \subset L$ be an extension of algebraically closed fields such that $L$ has transcendence degree $1$ over $K$. Let $V$ be a variety over $K$ and let $W$ be a subvariety of $V_L$. Then there exists a subvariety $W'$ of $V$ with $\dim W' \leq \dim W + 1$ such that $W \subset W'_L$.
\end{lem}

\begin{proof}We can replace $L$ by a finitely generated subextension of $K$ of transcendence degree $1$ over which $W$ is defined. We can then find a curve $C$ over $K$ such that $K(C) = L$. We can also find a subvariety $\mathcal{W} \subset V \times_K C$ of dimension $\dim W + 1$ such that the generic fiber of $\mathcal{W}$ over $C$ is $W$. The closure of the projection of $\mathcal{W}$ onto $V$ is our $W'$.
\end{proof}

We now give a formal definition of the trace of an abelian variety with respect to an extension of algebraically closed fields.

\begin{defn}\label{def:trace} \upshape (\cite{MR2255529}, Theorem 6.2) \itshape
 Let $K \subset L$ be an extension of algebraically closed fields. Let $A$ be an abelian variety defined over $L$. The $L/K$-trace of $A$ is a pair $(T,\Tr)$ of an abelian variety $T$ that is defined over $K$ and a homomorphism of algebraic groups $\Tr: T_L \to A$ that is characterized uniquely by the fact that for every abelian variety $T'$ that is defined over $K$ and every homomorphism of algebraic groups $\sigma: T'_L \to A$, there is a homomorphism of algebraic groups $\tau: T' \to T$ such that $\sigma = \Tr \circ \tau_L$. The map $\Tr$ is a closed embedding.
 \end{defn}

We introduce a notation that is going to make the exposition more agile.

\begin{defn}
	Let $K$ be an algebraically closed field and let $A$ be an abelian variety over $K$. Let $m$ and $d$ be non-negative integers. We say that $\ZP (A,m,d)$ holds if Conjecture \ref{conj:zilberpink} holds for $d$ and for all subvarieties $V$ of $A$ with $\dim V\leq m$. 
\end{defn}

We can then rephrase Theorem \ref{thm:functionfieldzp} as the implication
\[ \ZP (T,m,d) \implies \ZP (A,m,d)\]
for non-negative integers $m$ and $d$ and an abelian variety $A$ over an algebraically closed field $K$ with $K / \bar{ \mathbb{Q}}$-trace $(T,\Tr)$.

\subsection{Smoothness and optimality under homomorphisms}

We are going to project on quotients of abelian varieties by abelian subvarieties several times. The following two lemmata are going to be useful in this respect:

\begin{lem}\label{lem:smooth}
Let $K$ be an algebraically closed field and let $f: V \to W$ be a dominant morphism of algebraic varieties, defined over $K$. Then there exists $V_0 \subset V$ open and dense such that $f(V_0)$ is open and dense in $W$ and $f|_{V_0}: V_0 \to f(V_0)$ is smooth.
\end{lem}

\begin{proof}
By Corollary II.8.16 in \cite{MR0463157}, we can find $V_1 \subset V$ open, dense, and non-singular. By Theorem 10.19 in \cite{MR2675155}, $f(V_1)$ contains $V_2$ open and dense in $\overline{f(V_1)} = \overline{f(V)} = W$. By generic smoothness (Corollary III.10.7 in \cite{MR0463157}), we can find $V_3 \subset V_2$ open and dense in $V_2$ and hence in $W$ such that $V_0 = f|_{V_1}^{-1}(V_3)$ is open and dense in $V$ and $f|_{V_0}: V_0 \to V_3 = f(V_0)$ is smooth.
\end{proof}

\begin{lem}\label{lem:isog}
Let $K$ be an algebraically closed field. Let $f:A\rightarrow A'$ be a homomorphism of algebraic groups between abelian varieties defined over $K$. Then the following hold:
\begin{enumerate}
\item Let $V$ be a subvariety of $A$. Suppose that $V_0 \subset V$ is open and dense in $V$ such that $f(V_0)$ is open and dense in $f(V)$ and $f|_{V_0}: V_0 \to f(V_0)$ is smooth. Let $W$ be a subvariety of $V$ that is optimal for $V$ in $A$ and intersects $V_0$. If $\langle W \rangle$ contains a translate of a component of $\ker f$, then $f(W)$ has defect at most $\delta(W)$ and is optimal for $f(V)$ in $A'$.
\item If $f$ has finite kernel and $\ZP(A',m,d)$ holds, then $\ZP(A,m,d)$ holds.
\end{enumerate}
\end{lem}

\begin{proof}
	For (1), let $W$ be optimal for $V$ in $A$ such that $W\cap V_0 \neq \emptyset$, where $f|_{V_0}: V_0 \to f(V_0)$ is smooth of relative dimension $n$. Let $U$ be a subvariety of $A'$ such that $f(W) \subset U \subset f(V)$ and $\delta(U) \leq \delta(f(W))$. Since $f(\langle W \rangle)$ contains $\langle f(W)\rangle$ and $\langle W \rangle$ contains a translate of a component of $\ker f$, we have
	$$
	\delta(U)\leq \delta(f(W))\leq \dim f( \langle W \rangle) - \dim f(W)= \dim \langle W \rangle - \dim \ker f - \dim f(W).
	$$	
	 Let now $U'$ be an irreducible component of $f^{-1}(U) \cap V = f|_V^{-1}(U)$ that contains $W$. We have $U' \cap V_0 \neq \emptyset$ since it contains $W \cap V_0$. Furthermore, $U' \cap V_0$ is an irreducible component of $f|_{V_0}^{-1}(U \cap f(V_0))$. Since $f|_{V_0}: V_0 \to f(V_0)$ is smooth of relative dimension $n$, it follows that $\dim U' = \dim(U' \cap V_0) = \dim(U \cap f(V_0)) + n = \dim U + n$. Taking $U = f(W)$ and using that $W \subset U'$ shows that $\dim W \leq \dim f(W) + n$. If $U$ is again an arbitrary subvariety as above, we deduce that
	\begin{equation}\label{eq:defect}
	\delta(U)\leq \delta(f(W))\leq \dim \langle W \rangle - \dim \ker f - \dim W +n.
	\end{equation}
	After noting that $n\leq \dim \ker f$, we deduce that the defect of $f(W)$ is at most the defect of $W$.
	
	Assume now that $f(W)$ is not optimal for $f(V)$ in $A'$. It follows that we can choose $U$ as above with $f(W) \subsetneq U$. Since $\dim U > \dim f(W)$, we have $W \subsetneq U'$ and so $\delta(U') > \delta(W)$ by the optimality of $W$. Moreover, we certainly have $\dim \langle U' \rangle \leq \dim \langle U\rangle + \dim \ker f$, so it follows that $\delta(U') \leq \delta(U) + \dim \ker f - n$. Using \eqref{eq:defect}, we deduce that
$$
	\delta(U)\leq \delta(W) - \dim \ker f  +n< \delta(U')- \dim \ker f  +n \leq \delta(U),
	$$	 
	a contradiction. 

	We can now prove (2) by induction on $m$. For the base step of the induction, note that $\ZP(A,0,d)$ holds trivially for all $d$. Let now $V$ be a subvariety of $A$ of dimension $m$ and let $W$ be an optimal subvariety for $V$ in $A$ such that $W$ has defect at most $d$. By Lemma \ref{lem:smooth}, we can find $V_0 \subset V$ open and dense such that $f(V_0)$ is open and dense in $f(V)$ and $f|_{V_0}: V_0 \to f(V_0)$ is smooth.
	
	If $W \subset V\backslash V_0$, then $W$ is contained in one of the finitely many components of $V\backslash V_0$. The dimension of that component is at most $m-1$ and $W$ has defect at most $d$ and is optimal for that component in $A$, so we are done by induction. Otherwise, we have $W \cap V_0 \neq \emptyset$. Since $\ker f$ is finite, a translate of one of its components is trivially contained in $\langle W \rangle$, so we can apply (1) to find that $f(W)$ is optimal for $f(V)$ in $A'$ and has defect at most $d$. As $f$ has finite kernel, we have $\dim f(V) = \dim V$, so it follows from $\ZP(A',m,d)$ that $f(W)$ belongs to a finite set of varieties. The same follows for $W$ since $W$ is an irreducible component of $f^{-1}(f(W))$ because of the equality $\dim W = \dim f(W) = \dim f^{-1}(f(W))$.
\end{proof}

\subsection{Abelian schemes}

We start this subsection by associating to any abelian variety over an algebraically closed field of characteristic $0$ a ``subfamily'' of the universal family of principally polarized abelian varieties of fixed dimension with some fixed level structure.

 We denote by $A_{g,l}$ the moduli space of principally polarized abelian varieties of dimension $g$ with symplectic level $l$-structure over $\bar{\mathbb{Q}}$. For $l \geq 3$, this is a fine moduli space and we denote by $\mathfrak{A}_{g,l}$ the corresponding universal family over $A_{g,l}$.

\begin{lem}\label{lem:finemodulispace}
	Let $K$ be an algebraically closed field. Let $A$ be an abelian variety of dimension $g$ defined over $K$. Fix a natural number $l \geq 3$. Then there exists a subvariety $S \subset A_{g,l}$ with the following property: Let $\mathcal{A} = \mathfrak{A}_{g,l} \times_{A_{g,l}} S$ and let $A'$ be the generic fiber of $\mathcal{A} \to S$. There exists a field embedding $\bar{\mathbb{Q}}(S) \hookrightarrow K$ such that $A$ is isogenous to $A'_K$.
\end{lem}

\begin{proof}
	Over $K$, the abelian variety $A$ is isogenous to a principally polarized abelian variety by Corollary 1 on p. 234 of \cite{Mumford} and so we can assume without loss of generality that $A$ is itself principally polarized. We can find a field $K_0 \subset K$ and an abelian variety $B$ defined over $K_0$ such that $\bar{\mathbb{Q}} \subset K_0$, $K_0$ is finitely generated over $\bar{\mathbb{Q}}$, and $A = B_K$. Without loss of generality, we can assume that all torsion points of $B$ of order $l$ are $K_0$-rational. We can find a normal variety $V$, defined over $\bar{\mathbb{Q}}$, with $\bar{\mathbb{Q}}(V) = K_0$. By spreading out (see Theorem 3.2.1 and Table 1 on pp. 306--307 of \cite{MR3729254}), we find an abelian scheme $\mathcal{B} \to V$ with generic fiber $B$ (after maybe replacing $V$ by an open dense subset). The principal polarization of $B$ gives a principal polarization of $\mathcal{B} \to V$ by the argument on p. 6 of \cite{MR1083353}. Among the $l$-torsion points of $B$, we choose a symplectic basis with respect to the Weil pairing induced by the principal polarization. The elements of this basis extend to $l$-torsion sections of $\mathcal{B} \to V$. In this way, $\mathcal{B} \to V$ becomes a principally polarized abelian scheme with symplectic level $l$-structure as defined in the Appendix to Chapter 7 of \cite{MR1304906}.
	
	Since $A_{g,l}$ is a fine moduli space by the Appendix to Chapter 7 of \cite{MR1304906}, we get a morphism $\iota: V \to A_{g,l}$, defined over $\bar{\mathbb{Q}}$, such that $\mathcal{B} $ is isomorphic to $ \mathfrak{A}_{g,l} \times_{A_{g,l}} V$. Let $S$ be the closure of $\iota(V)$, let $\mathcal{A} = \mathfrak{A}_{g,l} \times_{A_{g,l}} S$ and let $A'$ be the generic fiber of $\mathcal{A}\to S$. It follows from the universal property of the fiber product that $\mathcal{B} $ is isomorphic to $ \mathcal{A} \times_{S} V$. The dominant morphism $\iota: V \to S$ yields a field embedding $\bar{\mathbb{Q}}(S) \hookrightarrow K_0$. Passing to the generic fiber over $V$ shows that $B$ is isomorphic to $A'_{K_0}$.
\end{proof}

Let $S$ be a variety over an algebraically closed field $K$. Given an abelian scheme $\mathcal{A} \to S$ with geometric generic fiber $A$, we need to extend abelian subvarieties and torsion points of $A$ and $K$-points of the trace of $A$ to abelian subschemes, torsion sections, and constant sections, possibly after a base change.

\begin{defn}\label{def:abeliansubscheme}
Let $S$ be a variety and $\mathcal{A} \to S$ an abelian scheme, both defined over an algebraically closed field $K$. A subgroup scheme (resp. an open subgroup scheme, resp. a closed subgroup scheme) $\mathcal{B}$ of $\mathcal{A}$ is a subscheme (resp. an open subscheme, resp. a closed subscheme) of $\mathcal{A}$ such that the zero section $S \to \mathcal{A}$ factors through $\mathcal{B}$ and the morphisms $\mathcal{B} \times_{S} \mathcal{B} \to \mathcal{A}$ and $\mathcal{B} \to \mathcal{A}$ that are induced by the addition and inversion morphism of $\mathcal{A}$ respectively factor through $\mathcal{B}$. An irreducible subgroup scheme $\mathcal{B}$ of $\mathcal{A}$ is called an abelian subscheme if $\mathcal{B}\to S$ is flat, proper, and dominant.
\end{defn}

Equivalently, an abelian subscheme is an irreducible closed subgroup scheme that is flat over $S$. An abelian subscheme $\mathcal{B}$ is itself an abelian scheme over $S$: For each natural number $N$, the multiplication-by-$N$ morphism from $\mathcal{B}$ to $\mathcal{B}$ is dominant and proper, hence surjective. It follows that the geometric fibers of $\mathcal{B}$ must be connected as desired.

\begin{lem}\label{lem:genericfiber}
Let $S$ be a normal variety and $\mathcal{A} \to S$ an abelian scheme, both defined over an algebraically closed field $K$. Let $\xi$ be the generic point of $S$. Suppose that every $l$-torsion point of $(\mathcal{A}_\xi)_{\overline{K(S)}}$ is $K(S)$-rational for some natural number $l \geq 3$, where $\overline{K(S)}$ denotes a fixed algebraic closure of $K(S)$. Then every abelian subvariety $B$ of $(\mathcal{A}_\xi)_{\overline{K(S)}}$ is the geometric generic fiber of an abelian subscheme $\mathcal{B} \subset \mathcal{A}$.
\end{lem}

\begin{proof}
Fix an abelian subvariety $B$ of $(\mathcal{A}_\xi)_{\overline{K(S)}}$. It is a consequence of the Poincar\'{e} reducibility theorem that there exists an endomorphism $\psi$ of $(\mathcal{A}_\xi)_{\overline{K(S)}}$ such that $B$ is the irreducible component of $\ker(\psi)$ containing the neutral element. By Theorem 2.4 in \cite{MR1154704}, every endomorphism of $(\mathcal{A}_\xi)_{\overline{K(S)}}$ is the base change of an endomorphism of $\mathcal{A}_\xi$. It follows that every abelian subvariety of $(\mathcal{A}_\xi)_{\overline{K(S)}}$ is the base change of an abelian subvariety of $\mathcal{A}_\xi$. We identify $\psi$ and $B$ with the corresponding endomorphism and abelian subvariety of $\mathcal{A}_\xi$ respectively.

By Theorem 3.2.1(iii) in \cite{MR3729254} (cf. Th\'eor\`eme 8.8.2(i) in \cite{EGA_4_3}), the endomorphism $\psi$ spreads out to a morphism from $\mathcal{A}_U = \mathcal{A} \times_S U$ to itself for $U \subset S$ open and dense. This morphism has to be an endomorphism by Corollary 6.4 on p.~117 of \cite{MR1304906}. By Proposition 2.7 in Chapter I of \cite{MR1083353}, as $S$ is normal, this endomorphism extends to an endomorphism $\Psi: \mathcal{A}\to \mathcal{A}$.

We get a closed subgroup scheme $\ker(\Psi)$ of $\mathcal{A}$. Let $\ker(\Psi)^0$ be the functor defined in Section 3 of Expos\'e VIB in \cite{SGA3}. We want to apply Corollaire 4.4 from Expos\'e VIB in \cite{SGA3} by verifying condition (ii).

Fix $s \in S$. The fiber $\ker(\Psi)_s$ is an algebraic group over a field (of characteristic $0$ as always) and hence smooth. Since $\Psi$ is proper, $\Psi(\mathcal{A})$ is a closed irreducible subscheme of $\mathcal{A}$ and it follows from the Fiber Dimension Theorem (Lemma 14.109 in \cite{MR2675155}) applied to the morphism $\Psi(\mathcal{A}) \to S$ that $\dim \Psi(\mathcal{A}_s) \geq \dim \Psi(\mathcal{A}_\xi)$. Similarly, we have $\dim \ker(\Psi)_s \geq \dim B$. Furthermore, we have $\dim \mathcal{A}_s = \dim \ker(\Psi)_s + \dim \Psi(\mathcal{A}_s)$. But then it follows from $\dim \Psi(\mathcal{A}_\xi) + \dim B = \dim \mathcal{A}_\xi = \dim \mathcal{A}_s$ that $\dim \ker(\Psi)_s = \dim B$. This means that the function $s \mapsto \dim \ker(\Psi)_s$ is constant on $S$.

Therefore, we can apply Corollaire 4.4 of Expos\'e VIB in \cite{SGA3} to find that $\ker(\Psi)^0$ is represented by an open subgroup scheme of $\ker(\Psi)$, which we also denote by $\ker(\Psi)^0$. By the same Corollaire, $\ker(\Psi)^0$ is smooth over $S$. By definition, the generic fiber of $\ker(\Psi)^0$ is equal to $B$.

The number of geometrically irreducible components of the fibers of $\ker(\Psi)$ is uniformly bounded. Therefore, for some large $N$, we have that $\ker(\Psi)^0$ is equal to the image of $\ker(\Psi)$ under the multiplication-by-$N$ morphism. As this morphism is proper, it follows that $\ker(\Psi)^0$ is closed in $\mathcal{A}$ and therefore the morphism $\ker(\Psi)^0 \to S$ is proper. Since $\ker(\Psi)^0$ is smooth over $S$, it is flat over $S$. As its generic fiber is irreducible and it is flat over $S$, $\ker(\Psi)^0$ is irreducible as well by Proposition 2.3.4(iii) in \cite{MR0199181} and Section 2.1.8 of Chapter 0 of \cite{EGA_I}. Since $\ker(\Psi)^0$ is a subgroup scheme that dominates $S$, this shows that $\ker(\Psi)^0$ is an abelian subscheme of $\mathcal{A}$ with generic fiber equal to $B$ as desired.
\end{proof}

Lemma \ref{lem:genericfiber} does not hold if the base variety $S$ and the abelian scheme $\mathcal{A}$ are allowed to be arbitrary. We provide a non-trivial counterexample: We let $S \subset \mathbb{A}_K^1\backslash\{0,1\} \times_K \mathbb{A}_K^1$ be defined by the equation $(\lambda+1)^2\lambda = \mu^2$ in the affine coordinates $(\lambda,\mu)$ on $\mathbb{A}_K^2$. Let $\xi$ be the generic point of $S$. We consider two elliptic schemes $\mathcal{E}$ and $\mathcal{E}'$ over $S$ that are defined in $\mathbb{P}_K^2 \times_K S$ by equations $y^2z = x(x-z)(x-\lambda z)$ and $\lambda y'^2z' = x'(x'-z')(x'-\lambda z')$ in the projective coordinates $[x:y:z]$ and $[x':y':z']$ respectively. We set $\mathcal{A} = \mathcal{E} \times_S \mathcal{E}'$ and let $B$ be the abelian subvariety of $\mathcal{A}_\xi \subset \mathbb{P}_{K(S)}^2 \times_{K(S)} \mathbb{P}_{K(S)}^2$ defined by the equations $xz' = x'z$ and $yz' = \frac{\mu}{\lambda+1}y'z$. If $p = (-1,0) \in S(K) \subset \mathbb{A}_K^2(K) = K^2$ is the singular point of $S$, then $B$ extends to an abelian subscheme of $\mathcal{A} \times_S (S\backslash\{p\})$, but the fiber over $p$ of the closure of $B$ in $\mathcal{A}$ has two irreducible components.

\begin{lem}\label{lem:torsionconstant}

Let $S$ be a normal variety and $\mathcal{A} \to S$ an abelian scheme, both defined over an algebraically closed field $K$. Let $\xi$ be the generic point of $S$. Suppose that every $l$-torsion point of $(\mathcal{A}_\xi)_{\overline{K(S)}}$ is $K(S)$-rational for some natural number $l \geq 3$, where $\overline{K(S)}$ denotes a fixed algebraic closure of $K(S)$. Let $N$ be a fixed natural number.

Then there exists a normal variety $S'$ over $K$ with generic point $\eta$ and a finite surjective \'etale morphism $S' \to S$ such that the following hold for $\mathcal{A}' = \mathcal{A} \times_S S'$ (after fixing an algebraic closure $\overline{K(S')}$ of $K(S')$):

\begin{enumerate}
\item Every torsion point of $(\mathcal{A}'_\eta)_{\overline{K(S')}}$ of order $N$ is the geometric generic fiber of a torsion section $S' \to \mathcal{A}'$.
\item Let $(T,\Tr)$ denote the $\overline{K(S')}/K$-trace of $(\mathcal{A}'_\eta)_{\overline{K(S')}}$. Then $\Tr\left(T_{\overline{K(S')}}\right)$ is the geometric generic fiber of an abelian subscheme $\mathcal{T}$ of $\mathcal{A}'$ that is isomorphic (as an abelian scheme) to $T \times_K S'$.
\end{enumerate}

\end{lem}

\begin{proof}
Let $\epsilon: S \to \mathcal{A}$ be the zero section, let $g$ denote the relative dimension of $\mathcal{A} \to S$ and let $[N]: \mathcal{A} \to \mathcal{A}$ denote the multiplication-by-$N$ morphism. The morphism $[N]^{-1}(\epsilon(S)) \to S$ is finite and \'etale since $\epsilon$ is a closed embedding and $[N]$ is finite and \'etale by Proposition 20.7 in \cite{Milne}.

Using Proposition 2.3.4(iii) in \cite{MR0199181}, we deduce that every irreducible component of $[N]^{-1}(\epsilon(S)) \to S$ surjects onto $S$. Furthermore, $[N]^{-1}(\epsilon(S))$ is normal by \cite{stacks-project}, Tag 033C, as it is \'etale over the normal variety $S$. Hence, no two distinct irreducible components of $[N]^{-1}(\epsilon(S)) \to S$ can intersect each other. Therefore, every irreducible component of $[N]^{-1}(\epsilon(S))$ is finite and \'etale over $S$.

We can then achieve (1) by successively base changing to irreducible components of $[N]^{-1}(\epsilon(S))$ which have degree $> 1$ over $S$ and using at the end that a finite \'etale morphism of degree $1$ is an isomorphism. We take $S'$ to be the base variety we end up with.

For (2), we then first use Lemma \ref{lem:genericfiber} to identify $\Tr\left(T_{\overline{K(S')}}\right)$ with the geometric generic fiber of an abelian subscheme $\mathcal{T}$ of $\mathcal{A}'$. Second, we note that $\Tr$ is the base change of a homomorphism $T_{K(S')} \to \mathcal{A}'_\eta$ by Theorem 2.4 in \cite{MR1154704}. We denote this homomorphism also by $\Tr$. Third, arguing as in the proof of Lemma \ref{lem:genericfiber}, we can extend the induced isomorphism of abelian varieties $T_{K(S')} \to \Tr\left(T_{K(S')}\right)$ to an isomorphism of abelian schemes $T \times_K S' \to \mathcal{T}$.
\end{proof}

\subsection{Defect and optimality}

Here we show that extending the field of definition cannot give new optimal subvarieties.

\begin{lem}\label{lem:basechange}
	Let $K \subset L$ be an extension of algebraically closed fields. Let $A$ be an abelian variety defined over $K$ and let $V$ be a subvariety of $A$. If $W$ is an optimal subvariety for $V_L$ in $A_L$, then there exists an optimal subvariety $W'$ for $V$ in $A$ such that $W = (W')_L$ and $\delta(W) = \delta(W')$.
\end{lem}

\begin{proof}
	Since $A$ is defined over $K$, which is algebraically closed, we have that any special subvariety of $A_L$ is the base change of a special subvariety of $A$. Therefore, if $V$ is a subvariety of $A$, any optimal subvariety for $V_L$ in $A_L$ is an irreducible component of an intersection $V_L\cap H_L$ for some special subvariety $H$ of $A$ and is then the base change of a subvariety $W\subset V$ that must be optimal for $V$ in $A$ and of the same defect.
\end{proof}

We also introduce a new kind of defect.

\begin{defn}\label{def:defectmsv} Let $K \subset L$ be an extension of algebraically closed fields. Let $A$ be an abelian variety defined over $L$ with $L/K$-trace $(T,\Tr)$.
For a subvariety $V$ of $A$ we define 
 $\langle V\rangle_{K,geo}$ to be the smallest translate of an abelian subvariety of $A$ by a point in $\Tr(T({K}))+A_{\tors}$ that contains $V$. We call $K$-geodesic defect the difference $\delta_{K,geo}(V)=\dim \langle V \rangle_{K,geo} -\dim V$. If $W\subset V \subset A$, we say that $W$ is $K$-geodesic-optimal for $V$ in $A$ if $\delta_{K,geo}(U) > \delta_{K,geo}(W)$ for every subvariety $U$ with $W \subsetneq U \subset V$. 
\end{defn}

In the case $K=L$, we drop $K$ in the notation as we are considering usual weakly special subvarieties, the geodesic defect, and geodesic optimality as defined in \cite{MR3552014}.

Note that $A$ is isogenous to $T_L \times B$ for an abelian variety $B$ such that there exists no non-trivial homomorphism between $T_L$ and $B$. Therefore the intersection of two translates of abelian subvarieties of $A$ by points in $\Tr(T({K}))+A_{\tors}$ is again a finite union of translates of abelian subvarieties of $A$ by points in $\Tr(T({K}))+A_{\tors}$, so $\langle V \rangle_{K,geo}$ is well defined.

In \cite{MR3552014}, Habegger and Pila defined the defect condition for subvarieties of complex abelian varieties. If ${W}$ and ${V}$ are subvarieties of $A$ such that ${W} \subset {V}$, then the condition says that $\delta({V}) - \delta_{geo}({V}) \leq \delta({W}) - \delta_{geo}({W})$. They then showed that the fact that this holds in this setting implies that optimal subvarieties are also geodesic-optimal. Here we do the same for $\delta_{K,geo}$ in place of $\delta_{geo}$ and show that optimal subvarieties are also $K$-geodesic-optimal.

\begin{lem}\label{lem:defectcondition}
	In the setting of Definition \ref{def:defectmsv}, let ${W}$ and ${V}$ be subvarieties of $A$ such that ${W} \subset {V}$. Then, the following hold:
	\begin{enumerate}
		\item We have $\delta({V}) - \delta_{K,geo}({V}) \leq \delta({W}) - \delta_{K,geo}({W})$.
		\item If ${W} \subset {V}$ is optimal for ${V}$ in $A$, then it is $K$-geodesic-optimal for ${V}$ in $A$.
	\end{enumerate}
\end{lem}

\begin{proof}
	
	The deduction of (2) from (1) is done analogously to the proof of Proposition 4.5 in \cite{MR3552014}.
	
	It remains to prove (1): the defect condition in this case amounts to proving that
	\[\dim \langle V \rangle - \dim \langle V \rangle_{K,geo} \leq \dim \langle W \rangle - \dim \langle W \rangle_{K,geo}.\]
	For this, we can just copy the proof of Proposition 4.3(ii) in \cite{MR3552014}.
\end{proof}

The geodesic defect of a subvariety of the connected mixed Shimura variety $(\mathfrak{A}_{g,l})_{\mathbb{C}}$, which we define below, is linked to the $\mathbb{C}$-geodesic defect of the components of its geometric generic fiber.

\begin{lem}\label{lem:defectdefinition}
Suppose that $\mathcal{U}$ is a subvariety of $(\mathfrak{A}_{g,l})_{\mathbb{C}}$ and $S$ is its image under the projection to $(A_{g,l})_{\mathbb{C}}$. Let $U$ be an irreducible component of the fiber of $\mathcal{U}$ over the geometric generic point of $S$. We define the geodesic defect $\delta_{geo}(\mathcal{U})$ of $\mathcal{U}$ to be $\dim \langle S \rangle_{geo} - \dim S + \dim \langle U \rangle_{\mathbb{C},geo}-\dim U$, where $\langle S \rangle_{geo}$ is the smallest weakly special subvariety of $(A_{g,l})_{\mathbb{C}}$ (see \cite{UY11} for a definition and a characterization) that contains $S$.
This definition of geodesic defect is independent of the choice of $U$ and agrees with $\delta_{ws}$ as defined in Definition 8.1(i) in \cite{G18}.
\end{lem}

\begin{proof}
The independence from the choice of $U$ follows from the fact that the irreducible components of the fiber of $\mathcal{U}$ over the geometric generic point of $S$ form one orbit under the action of the Galois group $\Gal\left(\overline{\mathbb{C}(S)}/\mathbb{C}(S)\right)$ and $\langle \cdot \rangle_{\mathbb{C},geo}$ commutes with the action of $\Gal\left(\overline{\mathbb{C}(S)}/\mathbb{C}(S)\right)$.

In Definition 8.1(i) in \cite{G18}, the defect $\delta_{ws}(\mathcal{U})$ is defined as $\dim \mathcal{U}^{biZar} - \dim \mathcal{U}$, where $\mathcal{U}^{biZar}$ is the smallest bi-algebraic subvariety of $(\mathfrak{A}_{g,l})_{\mathbb{C}}$ containing $\mathcal{U}$. By Proposition 5.3 in \cite{G182}, this is equal to $\dim \langle S \rangle_{geo} - \dim S + \dim \mathcal{W} -\dim \mathcal{U}$, where $\mathcal{W}$ is the smallest generically special subvariety of sg type (as defined in Definition 1.5 in \cite{G182}) of $(\mathfrak{A}_{g,l})_{\mathbb{C}} \times_{(A_{g,l})_{\mathbb{C}}} S$ containing $\mathcal{U}$.

It is enough to show that $\dim \langle U \rangle_{\mathbb{C},geo}-\dim U = \dim \mathcal{W} -\dim \mathcal{U}$. By looking at the geometric generic fiber of $\mathcal{W}$, we see that $\dim \langle U \rangle_{\mathbb{C},geo}-\dim U \leq \dim \mathcal{W} -\dim \mathcal{U}$. For the inequality in the other direction to hold, we need to know that after a finite surjective base change we can extend abelian subvarieties and torsion points of the geometric generic fiber and $\mathbb{C}$-points of the trace to abelian subschemes, torsion sections, and constant sections respectively. This is ensured by Lemma \ref{lem:genericfiber} and Lemma \ref{lem:torsionconstant}; note that after a finite surjective base change we can assume the base to be normal.
\end{proof}

\begin{defn}\label{def:geodesicoptimalmsv}
If $\mathcal{W} \subset \mathcal{V}$ are subvarieties of $(\mathfrak{A}_{g,l})_{\mathbb{C}}$, we say that $\mathcal{W}$ is geodesic-optimal for $\mathcal{V}$ in $(\mathfrak{A}_{g,l})_{\mathbb{C}}$ if $\delta_{geo}(\mathcal{U}) > \delta_{geo}(\mathcal{W})$ for every subvariety $\mathcal{U}$ with $\mathcal{W} \subsetneq \mathcal{U} \subset \mathcal{V}$.
 \end{defn}
 
\section{A statement in the universal family}
 
The following result is a fundamental tool for our proof. It relies on a result of Gao, which is formulated in the language of mixed Shimura varieties. We show that, in the special setting of an abelian variety that is the geometric generic fiber of a ``subfamily'' of the universal family, Gao's result yields a strengthening of what is sometimes called a ``Structure Theorem'' proven by R\'emond for abelian varieties in \cite{MR2534482}. The analogous statement for powers of the multiplicative group was proven by Poizat in Corollaire 3.7 in \cite{Poizat} and independently by Bombieri, Masser and Zannier in \cite{BMZ07}. Recall that $A_{g,l}$ is a variety over $\bar{\mathbb{Q}}$.
 
\begin{thm}\label{thm:gao}
Let $S$ be a subvariety of $A_{g,l}$ and $\mathcal{A} = \mathfrak{A}_{g,l} \times_{A_{g,l}} S$. Let $A$ be the geometric generic fiber of $\mathcal{A}$ and $V$ a subvariety of $A$ and let $(T,\Tr)$ be the $K/\bar{\mathbb{Q}}$-trace of $A$, where $K$ is a fixed algebraic closure of $\bar{\mathbb{Q}}(S)$. There is a finite set of pairs $(q_0,H)$, where $q_0 \in A(K)$ is a torsion point and $H$ is an abelian subvariety of $A$, and a finite union $Z$ of proper subvarieties of $V$ such that for every optimal $W \subset V$ one of the following holds:
\begin{enumerate}
\item $W$ is contained in $Z$, or
\item there exists some point $t\in \Tr(T(\bar{\mathbb{Q}}))$ such that $W$ is an irreducible component of $(t + q_0 + H) \cap V$ and $t+q_0+H \subset \langle W \rangle$.
\end{enumerate}
\end{thm}

\begin{proof}
We want to apply a result of Gao (Theorem 8.2 in \cite{G18}). For this, we need to make a base change.
	Let $K'$ be a fixed algebraic closure of the function field $\mathbb{C}(S_\mathbb{C})$ of $S_{\mathbb{C}}$. Both $K$ and $\mathbb{C}$ embed into $K'$ and the intersection of their images is $\bar{\mathbb{Q}}$ since the transcendence degree of $K'/\mathbb{C}$ is equal to the transcendence degree of $K/\bar{\mathbb{Q}}$.
	
Let $W$ be optimal for $V$ in $A$. We want to show that $W_{K'}$ is optimal for $V_{K'}$ in $A_{K'}$. If $U$ is an optimal subvariety for $V_{K'}$ containing $W_{K'}$ and satisfying $\delta(U) \leq \delta(W_{K'})$, then, by Lemma \ref{lem:basechange}, $U$ is the base change of a subvariety of $V$ of the same defect that is optimal for $V$ in $A$. Because of the optimality of $W$ for $V$ in $A$, that subvariety has to be equal to $W$, so $U = W_{K'}$.

By Lemma \ref{lem:defectcondition}, $W_{K'}$ is also $\mathbb{C}$-geodesic-optimal for $V_{K'}$ in $A_{K'}$.

Suppose first that $W_{K'} \subset \sigma(V_{K'})$ for some $\sigma \in \Gal(K'/\mathbb{C}(S_{\mathbb{C}}))$ with $\sigma(V_{K'}) \neq V_{K'}$. Then $W_{K'}$ is contained in $V_{K'} \cap \sigma(V_{K'}) \subsetneq V_{K'}$. Let $Z_{\sigma} \subset A$ be maximal among all finite unions of subvarieties $Z' \subset A$ with $Z'_{K'} \subset V_{K'}\cap \sigma(V_{K'})$ (and equal to the closure of the union of all such $Z'$). Then $Z_{\sigma}$ is a finite union of proper subvarieties of $V$ and contains $W$. We set $Z = \cup_{\sigma(V_{K'}) \neq V_{K'}}{Z_{\sigma}}$; the union is finite since $\sigma(V_{K'})$ varies in a finite set. We deduce that $W \subset Z$, so (1) is satisfied.

From now on, we assume that $W_{K'} \subset \sigma(V_{K'})$ only holds if $\sigma(V_{K'}) = V_{K'}$ (for $\sigma \in \Gal(K'/\mathbb{C}(S_{\mathbb{C}}))$), and we want to prove that (2) holds.

The subvarieties $W_{K'}$ and $V_{K'}$ are irreducible components of the base change of subvarieties of the generic fiber of $\mathcal{A}_{\mathbb{C}}$. We define $\mathcal{W}$ and $\mathcal{V}$ to be the closures of these two subvarieties in $\mathcal{A}_{\mathbb{C}}$. Note that they are subvarieties of dimension $\dim S +\dim W$ and $\dim S +\dim V$ respectively and that they dominate $S_\mathbb{C}$.
	
In Lemma \ref{lem:defectdefinition}, we defined the geodesic defect of subvarieties of $(\mathfrak{A}_{g,l})_\mathbb{C}$ and we have seen that it coincides with $\delta_{ws}$ of \cite{G18}. Let $\mathcal{U}\subset \mathcal{V}$ be a geodesic-optimal subvariety for $\mathcal{V}$ that contains $\mathcal{W}$ and satisfies $\delta_{geo}(\mathcal{U}) \leq \delta_{geo}(\mathcal{W})$. Using that $\mathcal{W} \subset \mathcal{U}$, we find that there exists an irreducible component $U$ of the geometric generic fiber of $\mathcal{U}$ that contains $W_{K'}$ and satisfies
\begin{multline*}
\delta_{\mathbb{C},geo}(U) = \delta_{geo}(\mathcal{U}) - \dim \langle S \rangle_{geo} + \dim S \\
\leq \delta_{geo}(\mathcal{W}) - \dim \langle S \rangle_{geo} + \dim S = \delta_{\mathbb{C},geo}(W_{K'}).
\end{multline*}
Since $\mathcal{U} \subset \mathcal{V}$, we can deduce that $U \subset \sigma(V_{K'})$ for some $\sigma \in \Gal(K'/\mathbb{C}(S_{\mathbb{C}}))$. Since $W_{K'} \subset U \subset \sigma(V_{K'})$, we must have $\sigma(V_{K'}) = V_{K'}$ by our assumption from above.

Since $W_{K'}$ is $\mathbb{C}$-geodesic-optimal for $V_{K'}$ in $A_{K'}$, it follows that $W_{K'} = U$ and therefore $\mathcal{W} = \mathcal{U}$. Hence, $\mathcal{W}$ is geodesic-optimal for $\mathcal{V}$ in the connected mixed Shimura variety $(\mathfrak{A}_{g,l})_{\mathbb{C}}$.

Let $\mathcal{W}^{biZar}$ be the smallest bi-algebraic subvariety of $(\mathfrak{A}_{g,l})_{\mathbb{C}}$ that contains $\mathcal{W}$. It is determined by a tuple $(Q,\mathcal{Y}^{+},N,\tilde{y})$, where $(Q,\mathcal{Y}^{+})$ is a connected mixed Shimura subdatum of $(\GSp_{2g} \ltimes \mathbb{Q}^{2g},\mathbb{H}_{g} \times \mathbb{C}^g)$, $N$ is a normal subgroup of the derived subgroup $Q^{\der}$, and $\tilde{y} \in \mathcal{Y}^{+}$. Here $\mathbb{H}_g$ denotes the Siegel upper half space. Thanks to Theorem 8.2 in \cite{G18}, we know that the triple $(Q,\mathcal{Y}^{+},N)$ lies in a finite set that does not depend on $\mathcal{W}$. By Proposition 5.3 in \cite{G182}, $\mathcal{W}^{biZar}$ is a generically special subvariety of sg type (as defined in Definition 1.5 in \cite{G182}) of $(\mathfrak{A}_{g,l})_{\mathbb{C}} \times_{(A_{g,l})_{\mathbb{C}}} \pi(\mathcal{W}^{biZar})$, where $\pi: (\mathfrak{A}_{g,l})_{\mathbb{C}} \to (A_{g,l})_{\mathbb{C}}$ is the structural morphism, so up to finite surjective base change a translate of an abelian subscheme by a torsion section and a constant section.

Looking at the proof of Proposition 3.3 on p. 240 of \cite{G15}, we see that the abelian subscheme and the torsion section are uniquely determined by $(Q,\mathcal{Y}^{+},N)$. Note that $\tilde{y}_G$ in the proof of Proposition 3.3 in \cite{G15} can be assumed fixed as $\pi(\mathcal{W}) = S_{\mathbb{C}}$ is independent of $\mathcal{W}$. After intersecting $\mathcal{W}^{biZar}$ with $\mathcal{A}_{\mathbb{C}}$ and passing to the geometric generic fiber, we deduce from this together with Lemma \ref{lem:defectdefinition} that there exists a finite set of tuples $(q_0,H)$, where $q_0 \in A(K)$ is a torsion point and $H$ is an abelian subvariety of $A$, such that there exists some point $t\in \Tr'(T'(\mathbb{C}))$ with $\langle W_{K'}\rangle_{\mathbb{C},geo} = t + (q_0 + H)_{K'}$. Here, $(T',\Tr')$ is the $K'/\mathbb{C}$-trace of $A_{K'}$.

	First of all, $\Tr': T'_{K'} \to A_{K'}$ is the base change of a homomorphism $T'_{K\mathbb{C}} \to A_{K\mathbb{C}}$ by Theorem 2.4 in \cite{MR1154704}, because all torsion points of domain and codomain are $K\mathbb{C}$-rational. Hence, $(T',\Tr')$ is equal to the base change of the $K\mathbb{C}/\mathbb{C}$-trace of $A_{K\mathbb{C}}$. Furthermore, this latter trace is equal to $(T_{\mathbb{C}},\Tr_{K\mathbb{C}})$ by Theorem 6.8 in \cite{MR2255529}. It follows that $T' = T_{\mathbb{C}}$ and $\Tr' = \Tr_{K'}$.
	
	Since it is $\mathbb{C}$-geodesic-optimal for $V_{K'}$ in $A_{K'}$, $W_{K'}$ itself must be equal to an irreducible component of $(t+(q_0 + H)_{K'}) \cap V_{K'}$.
	
	We now want to show that we can take $t$ to be the image of the base change of a $\bar{\mathbb{Q}}$-rational point of $T$.
	Indeed, the image of any point in $X=\Tr_{K'}^{-1} \left(W_{K'}+(-q_0 + H)_{K'}\right) $ that is the base change of a $\mathbb{C}$-rational point of $T_{\mathbb{C}}$ can be chosen as $t$. The finite union of subvarieties $X$ is equal to the base change of $\Tr^{-1}(W+(-q_0+H)) \subset T_{K}$. On the other hand, one can see that $X$ is equal to $\Tr_{K'}^{-1}\left(t+H_{K'}\right)$. Since $\Tr$ is a homomorphism and every algebraic subgroup of $T_{K'}$ is the base change of an algebraic subgroup of $T$, this means that $X$ is the base change of a union of translates of an abelian subvariety of $T_{\mathbb{C}}$ by a point in $T_{\mathbb{C}}(\mathbb{C})$. Since $\mathbb{C} \cap K = \bar{\mathbb{Q}}$, it follows from Corollaire 4.8.11 in \cite{MR0199181} that $X$ is equal to the base change of a union of algebraic subvarieties of $T$ and $t$ can be chosen as the image of the base change of a point of $T(\bar{\mathbb{Q}})$. If we denote this point also by $t$, we have that $W$ is an irreducible component of $(t+q_0+H) \cap V$. Since $(t + q_0 + H)_{K'} = \langle W_{K'} \rangle_{\mathbb{C},geo} \subset \langle W \rangle_{K'}$, it also follows that $t+q_0+H \subset \langle W \rangle$.
\end{proof}

We now apply Theorem \ref{thm:gao} to our problem.

\begin{prop}\label{prop:mixedaxschanuel}
	Let $S \subset A_{g,l}$ be a subvariety of positive dimension. Let $\mathcal{A} = \mathfrak{A}_{g,l} \times_{A_{g,l}} S$ and let $A$ be the geometric generic fiber of $\mathcal{A} \to S$. If $\ZP (B, m,d)$ holds for all quotients $B$ of $A$ by a positive-dimensional abelian subvariety, then $\ZP (A, m,d)$ holds.
\end{prop}

\begin{proof}
	We induct on $m$. Clearly $\ZP(A,0,d)$ holds for all $d$.
	
	Let $V$ be a subvariety of $A$ of dimension $m$ and let $W \subset V$ be an optimal subvariety of defect at most $d$. Let $(T, \Tr)$ denote the $K/\bar{\mathbb{Q}}$-trace of $A$, where $K$ is an algebraic closure of $\bar{\mathbb{Q}}(S)$. We have $\Tr(T_K) \neq A$ since $\dim S > 0$.
	
	We apply Theorem \ref{thm:gao}. If $W$ satisfies (1), then $W$ is contained in a component of $Z$ and optimal for that component. Furthermore, $W$ has defect at most $d$, so we are done by induction on $m$. If $W$ satisfies (2), then $W$ is an irreducible component of $(t+q_0+H) \cap V$, where $t \in \Tr(T(\bar{\mathbb{Q}}))$ and $(q_0,H)$ lies in a finite set of pairs of torsion points and abelian subvarieties of $A$ that does not depend on $W$. We can assume $H$ and $q_0$ fixed. We now quotient out by $H$. Let $f: A \to A/H$ be the corresponding morphism. We get a subvariety $f(V)$ of $A/H$ and a point $w$ such that $\{w\} = f(W)$. By Lemma \ref{lem:smooth}, we can find $V_0 \subset V$ open and dense such that $f(V_0)$ is open and dense in $f(V)$ and $f|_{V_0}: V_0 \to f(V_0)$ is smooth of some relative dimension $n$.
	
	If $W \subset V\backslash V_0$, then $W$ is contained in one of finitely many subvarieties of $V$ of dimension at most $m-1$ (and of course optimal for that subvariety in $A$ and of defect at most $d$), so we are done by induction. Hence we can assume that $W \cap V_0 \neq \emptyset$. Since $W \cap V_0$ is an irreducible component of $f|_{V_0}^{-1}(\{w\})$, it follows that $n = \dim(W \cap V_0) = \dim W$.
	
	If $\dim H = 0$, then $W = \{t+q_0\}$. So the singleton $\{t\}$ is contained in some irreducible component $V'$ of $\Tr(T_K) \cap (-q_0+V)$. It has defect at most $d$ and is optimal for $V'$ in $\Tr(T_K)$. Since $\Tr(T_K) \neq A$, there is an isogeny between $\Tr(T_K)$ and a quotient of $A$ by some positive-dimensional abelian subvariety. We can use our hypothesis and Lemma \ref{lem:isog}(2) to deduce that $t$ and therefore $W$ belongs to a finite set.
	Hence, we can assume that $\dim H > 0$.

	By Theorem \ref{thm:gao}, we have $t + q_0 + H \subset \langle W \rangle$ and therefore $\langle W \rangle$ contains a translate of a component of $\ker f = H$. It follows from Lemma \ref{lem:isog}(1) that $\{w\}$ has defect at most $d$ and is optimal for $f(V)$ in $A/H$. Now we can use that $\ZP(A/H,m,d)$ holds to deduce that $w$ and hence $W$ as an irreducible component of $f^{-1}(\{w\}) \cap V$ must lie in a finite set.
\end{proof}

\section{Reduction of the transcendence degree}

The following proposition gives us the final reduction to the algebraic case or to what we proved in Proposition \ref{prop:mixedaxschanuel}.

\begin{prop}\label{prop:fieldofdefofcurve}
	Let $m$ and $d$ be non-negative integers. Let $K\subset L$ be an extension of algebraically closed fields. Let $A$ be an abelian variety defined over $K$. If $\ZP (A,m,d)$ holds, then $\ZP (A_L,m,d)$ holds as well.
\end{prop}

\begin{proof}
	Let $V$ be a subvariety of $A_L$ of dimension at most $m$. We can find an algebraically closed subfield $L_1$ of $L$ that has finite transcendence degree over $K$ and a subvariety $V_1$ of $A_{L_1}$ such that $V = (V_1)_{L}$. If $W$ is any optimal subvariety for $V$ in $A_L$, then by Lemma \ref{lem:basechange} it is equal to $(W_1)_L$ for an optimal subvariety $W_1$ for $V_1$ in $A_{L_1}$ such that $\delta(W_1) = \delta(W)$. Hence, it suffices to prove the proposition under the assumption that $L$ has finite transcendence degree over $K$.

	Arguing by induction on the transcendence degree of $L$ over $K$, one can see that it is enough to prove our statement when $L$ has transcendence degree 1 over $K$.
	
	We proceed by induction on $m$. Clearly $\ZP(A_L,0,d)$ holds for all $d$, so, for some positive $m$, we will deduce $\ZP (A_L, m,d)$ from $\ZP (A_L, m-1,d)$ and $\ZP (A,m,d)$.
	
	Let $V$ be a subvariety of $A_L$ of dimension $m$.	
	If $V=V'_L$ for some $V'\subset A$ then we are done by Lemma \ref{lem:basechange} and $\ZP(A,m,d)$. We will then assume that this is not the case. 
	
	Let $V'_L$ be the smallest subvariety of $A_L$ that is the base change of some $V'\subset A$ and contains $V$. It exists and has dimension $m$ or $m+1$ by Lemma \ref{lem:ausefullemma} but the first case is not possible because it would imply that $V=V'_L$.
	
	Let $W \subset V$ be an optimal subvariety for $V$ in $A_L$ that has defect at most $d$. We can assume without loss of generality that $W \neq V$.
	
	We let $W'_L$ be the smallest subvariety of $A_L$ that is the base change of some $W'\subset A$ and contains $W$. By Lemma \ref{lem:ausefullemma}, we have either $W = W'_L$ or $\dim W'_L = \dim W + 1$.

	If $W = W_L'$, then $W$ is contained in $Z'_L \subset V$ for $Z'\subset A$ maximal among all finite unions of subvarieties $Z'' \subset A$ with $Z''_L \subset V$ (and equal to the closure of the union of all such $Z''$). Since $V\neq V'_L$, the dimension of $Z'_L$ is at most $m-1$. Of course, $W$ is also optimal for the component of $Z'_L$ that contains it and therefore lies in a finite set because $\ZP(A_L, m-1,d)$ holds. We can therefore assume that $W \subsetneq W'_L$ and so $\dim W'_L = \dim W + 1$.
	
	Recall that, by Lemma \ref{lem:basechange}, an optimal subvariety for $V'_L$ in $A_L$ is the base change of an optimal subvariety for $V'$ in $A$. Let $U'_L$ be such an optimal subvariety for $V'_L$ in $A_L$ that contains $W'_L$ and satisfies $\delta(U'_L) \leq \delta(W'_L)$. Note that $\langle W \rangle = \langle W'_L\rangle$ and $\langle V \rangle = \langle V'_L \rangle$ because, for instance, $V'_L\subset\langle V \rangle \cap V'_L $ by definition. It follows that $\delta(W'_L) = \delta(W)-1$, so $\delta(U'_L) \leq d-1$.
	
	We claim that $U'_L \neq V'_L$. If not, we could deduce that $\delta(V'_L) = \delta(U'_L) \leq \delta(W'_L)$. It would then follow that $\delta(V) = \delta(V'_L)+1 \leq \delta(W'_L)+1 = \delta(W)$, which contradicts the optimality of $W \subsetneq V$ for $V$. 
	
	We deduce that $U'_L \subsetneq V'_L$ and hence $U'_L \cap V \subsetneq V$, otherwise $U'_L\supset V$ would contradict the minimality of $V'_L$. Since $W \subset U'_L \cap V$ and $W$ has defect at most $d$ and is optimal for a component of $U'_L \cap V$ in $A_L$, it suffices to show that $U'$ and therefore $U'_L$ belongs to a finite set and then we are done as $\dim(U'_L \cap V) < \dim V$ and $\ZP(A_L,m-1,d)$ holds.
		
	It follows from the optimality of $U'$ for $V'$ in $A$ and Proposition 4.5 in \cite{MR3552014} that $U'$ is also geodesic-optimal for $V'$ in $A$. We can apply the results of R\'{e}mond in \cite{MR2534482} (the connection to geodesic optimality is explained in Section 6 of \cite{MR3552014}) to deduce that there exists a finite set of abelian subvarieties of $A$ such that for each geodesic-optimal $U$ for $V'$ in $A$ there exists $H$ in this finite set such that for any $u \in U(K)$ we have $\langle U \rangle_{geo} = u+H$ (and $U$ is an irreducible component of $(u+H) \cap V'$ since it is geodesic-optimal for $V'$).
	
	Since $H$ varies in a finite set, we can assume it fixed and divide out by it. Let $f: A \to A/H$ be the corresponding morphism. We get a subvariety $f(V')$ of $A/H$ and a point $u'\in (A/H)(K)$ such that $\{u'\} = f(U')$.
	
	By Lemma \ref{lem:smooth}, we can find $V'_0 \subset V'$ open and dense such that $f(V'_0)$ is open and dense in $f(V')$ and $f|_{V'_0}: V'_0 \to f(V'_0)$ is smooth of relative dimension $n = \dim V'_0 - \dim f(V'_0) = \dim V' - \dim f(V')$. If $U'$ is contained in $V'\backslash V'_0$, then $U'$ is contained in one of finitely many subvarieties of $V'$ of dimension at most $m$. As $U'$ is optimal for $V'$, it is also optimal for that subvariety. We have $\delta(U') \leq \delta(W') = \delta(W)-1 \leq d-1$. By $\ZP(A,m,d)$, there are then only finitely many possibilities for $U'$.
	
	Hence we can assume that $U' \cap V'_0 \neq \emptyset$. Since $U'$ is an irreducible component of $(u+H) \cap V'$ (for some $u \in U'(K)$), $U' \cap V'_0$ is then an irreducible component of $(u+ H) \cap V_0' = f|_{V_0'}^{-1}(\{u'\})$. It follows that $n = \dim(U' \cap V'_0) = \dim U'$.
	
	Since we know that $W\neq W'_L$, we have $\dim W' > 0$ and hence $n = \dim U' > 0$.
	
	Note that $\langle U' \rangle$ contains a translate of a component of $\ker f = H$ since $\langle U' \rangle_{geo} = u + H \subset \langle U' \rangle$ (for $u \in U'(K)$ arbitrary). It therefore follows from Lemma \ref{lem:isog}(1) that $\{u'\}$ is optimal for $f(V')$ in $A/H$ and has defect at most $d-1$. Furthermore, $f(V')$ is a subvariety of $A/H$ of dimension $\dim V' - n \leq \dim V' - 1 = m$.
	
Now, there is a homomorphism $A/H \rightarrow A$ of algebraic groups with finite kernel, so $\ZP(A/H,m,d)$ holds by Lemma \ref{lem:isog}(2). Thus, $u'$ lies in a finite set. As $U'$ is a component of $f^{-1}(\{u'\}) \cap V'$, it lies in a finite set as well.
\end{proof}

\section{Proof of Theorem \ref{thm:functionfieldzp}}

Fix non-negative integers $m$ and $d$. We argue by induction on the dimension of $A$. If the dimension of $A$ is at most $1$, the statement holds trivially. Let now $A$ be an abelian variety of dimension $> 1$ over an algebraically closed field $K$ and assume that Theorem \ref{thm:functionfieldzp} holds for the fixed $m$ and $d$ and all abelian varieties of smaller dimension. In particular, it holds for all quotients of $A$ by abelian subvarieties of positive dimension.

Applying Lemma \ref{lem:finemodulispace}, we find a subvariety $S$ of $A_{g,l}$ and an embedding of $\bar{\mathbb{Q}}(S)$ into $K$ such that $A$ is isogenous to $A'_{K}$, where $A'$ is the generic fiber of $\mathfrak{A}_{g,l}\times_{A_{g,l}} S$. 

Moreover, the $K/\bar{\mathbb{Q}}$-traces of $A$ and $A'_K$ are isogenous. Therefore, by Lemma \ref{lem:isog}(2), we only need to prove the statement of Theorem \ref{thm:functionfieldzp} for $A'_{K}$.

Proposition \ref{prop:fieldofdefofcurve} tells us that $\ZP(A'_K,m,d)$ follows from $\ZP (A'_{K'},m,d)$, where $K'$ is an algebraic closure of $\bar{ \mathbb{Q}}(S)$.

If $\dim S=0$, we have nothing to do. We then assume that $S$ has positive dimension.

By the inductive hypothesis we know that for all quotients $B$ of $A'_{K'}$ by a positive-dimensional abelian subvariety with $K'/\bar{ \mathbb{Q}}$-trace $(T_B, \Tr_B)$, the implication
$$
\ZP(T_B, m,d) \implies \ZP(B,m,d)
$$
holds.

Note that the $K/\bar{\mathbb{Q}}$-trace of $A'_K$ is equal to the base change of the $K'/\bar{\mathbb{Q}}$-trace $(T_{A'_{K'}},\Tr_{A'_{K'}})$ of $A'_{K'}$ by Theorem 6.4(3) in \cite{MR2255529}. If we know that $\ZP(T_{A'_{K'}}, m,d)$ holds, then, for all $B$ quotients of $A'_{K'}$, since there exists a homomorphism of algebraic groups with finite kernel from $T_{B}$ to $T_{A'_{K'}}$, $\ZP(T_{B}, m,d)$ holds as well because of Lemma \ref{lem:isog}(2).

The inductive hypothesis then tells us that $\ZP(B,m,d)$ holds for all quotients $B$ of $A'_{K'}$ by a positive-dimensional abelian subvariety and thus, by 
Proposition \ref{prop:mixedaxschanuel}, $\ZP(A'_{K'}, m,d)$ holds as we wanted to prove. \qed

\section{Proof of Theorem \ref{thm:pequalzp}} \label{sec:pequalzp}

We show that the hypotheses in Theorem \ref{thm:pequalzp} imply the following claim for every quotient $A'$ of $A$:

\begin{claim}\label{claim:first}
Every subvariety $V$ of $A'$ of dimension at most $m$ contains at most finitely many optimal singletons (for $V$ in $A'$) of defect at most $d$.
\end{claim}

As every quotient of $A$ admits a homomorphism of algebraic groups with finite kernel to $A$, we can deduce as in the proof of Lemma \ref{lem:isog}(2) that it suffices to prove Claim \ref{claim:first} for $A$ itself.

Let therefore $V$ be a subvariety of $A$ of dimension at most $m$. We show the following claim by induction on $j \in \{0,\hdots,\dim V\}$:

\begin{claim}\label{claim:second}
The optimal singletons for $V$ in $A$ of defect at most $d$ are contained in a finite union of subvarieties of $V$ of dimension at most $\dim V - j$.
\end{claim}

This is obvious for $j = 0$.

Suppose that Claim \ref{claim:second} has been proven for some $j < \dim V$. Let $W$ be one of the finitely many subvarieties of $V$ of dimension at most $\dim V - j$ that contain the optimal singletons for $V$ in $A$ of defect at most $d$. We can assume without loss of generality that $\dim W = \dim V -j$.

Any optimal singleton for $V$ in $A$ that is contained in $W$ is also optimal for $W$ in $A$. We want to show that the optimal singletons for $W$ in $A$ of defect at most $d$ are contained in a proper closed subset of $W$. This will establish Claim \ref{claim:second} for $j+1$.

Translating $W$ by a torsion point sends optimal singletons (for $W$ in $A$) to optimal singletons of the same defect, so we can assume without loss of generality that $B := \langle W \rangle$ is an abelian subvariety of $A$.

If $\{p\} \subset W$ is an optimal singleton for $W$ in $A$ of defect at most $d$, then $\langle \{p\} \rangle \subset B$. Since $\dim W = \dim V -j > 0$, we have that $\{p\} \subsetneq W$ and therefore $\delta(\{p\}) = \dim \langle \{p\} \rangle < \delta(W) = \dim B - \dim W$. It follows that the codimension of $\langle \{p\} \rangle$ in $B$ is greater than or equal to $k := \max\{\dim B-d, \dim W + 1\}$. So the optimal singletons for $W$ in $A$ of defect at most $d$ are contained in $W \cap B^{[k]}$.

As $B = \langle W \rangle$, no proper special subvariety of $B$ can contain $W$. It then follows from Conjecture \ref{conj:pink} for $B$, $d$, and $W$ that $W \cap B^{[k]}$ is not dense in $W$. Together with the above, this implies that the optimal singletons for $W$ in $A$ of defect at most $d$ are contained in a proper closed subset of $W$ as desired. This establishes Claim \ref{claim:second} by induction.

Now taking $j = \dim V$ shows that the number of optimal singletons for $V$ in $A$ of defect at most $d$ is finite. This proves Claim \ref{claim:first}. Theorem \ref{thm:pequalzp} now follows from the following theorem:

\begin{thm}\label{thm:reductiontosingletons}
Let $m$ and $d$ be non-negative integers and suppose that every subvariety of dimension at most $m$ of a quotient of $A$ contains at most finitely many optimal singletons of defect at most $d$. Then every subvariety of $A$ of dimension at most $m$ contains at most finitely many optimal subvarieties of defect at most $d$.
\end{thm}

\begin{proof}
The hypotheses imply that after fixing an arbitrary field of definition that is finitely generated over $\mathbb{Q}$, every quotient of $A$ satisfies $LGO^m_d$ as defined in Definition 8.1 in \cite{MR3552014}. Theorem \ref{thm:reductiontosingletons} then follows from Theorem 9.8(i) in \cite{MR3552014}.
\end{proof}

\section{The Zilber-Pink conjecture for subvarieties of defect $\leq 1$}

\begin{thm}\label{thm:habeggerpila}
Let $A$ be an abelian variety over $\bar{\mathbb{Q}}$ and $V \subset A$ a subvariety. Then $V$ contains at most finitely many optimal subvarieties of defect at most $1$.
\end{thm}

The following proof was suggested to the authors by Philipp Habegger.

\begin{proof}
Let $\{p\} \subset V$ be an optimal singleton for $V$ in $A$, contained in a torsion coset of dimension at most $1$. By Proposition 4.5 in \cite{MR3552014}, $\{p\}$ is geodesic-optimal for $V$ in $A$ and therefore not contained in a coset of positive dimension that is contained in $V$. By the Theorem in \cite{MR2495768} with $s=\dim A -1$, the height of $p$ with respect to any fixed symmetric ample line bundle on $A$ is then bounded.

It then follows from Proposition 9.7 in \cite{MR3552014} that (in the notation of \cite{MR3552014}) $LGO_1(V)$ is satisfied after fixing a number field over which $V$ and $A$ are defined. As $V$ was arbitrary, this implies that every abelian variety $A$ over $\bar{\mathbb{Q}}$ satisfies $LGO^m_1$ for all integers $m \geq 0$; see Definition 8.1 in \cite{MR3552014} for the definitions of $LGO_d(V)$ and $LGO^m_d$. The claim then follows from Theorem 9.8(i) in \cite{MR3552014}.
 \end{proof}

\section*{Acknowledgements}
We thank J\'er\'emy Blanc, Giulio Codogni, Ziyang Gao, Philipp Habegger, Lars K\"uhne, Immanuel van Santen, and Filippo Viviani for relevant and helpful discussions. We thank Philipp Habegger, Jonathan Pila, and Ga\"{e}l R\'emond for comments on a preliminary version of this article. This work was supported by the Swiss National Science Foundation as part of the project ``Diophantine Problems, o-Minimality, and Heights", no. 200021\_165525. The second-named author thanks the Dipartimento di Matematica e Fisica di Roma Tre for its hospitality during a very productive week spent working with the first-named author in Rome.

\bibliographystyle{amsalpha}
\bibliography{Bibliography.bib}
\end{document}